%% file: lncs_opt-ctrl.tex
\documentclass{llncs}[runningheads]
\usepackage{mysty}

\begin{document}

\title{ Controlled Optimization of Quadratic Functions in $\bbR^n$ }
\author{Jean-Jacques Godeme\inst{1}}
\institute{Inria  C\^ote d’Azur  \& Universit\'e C\^ote d’Azur, CNRS, LJAD, France. \underline{e-mail:}\texttt{jeanjacquesgodeme@gmail.fr.}
}
\maketitle
\begin{center}
	\textit{In honor to the 90th Birthday of Professor emiritus R. Tyrrell Rockafellar for all his contributions to mathematics and his role as an inspiration to younger generations.}
\end{center}%
\begin{abstract}$~$
In this work, we introduce and study the controllability of the trajectories of a linear dynamical system, which can be  used to solve the minimization of a quadratic function in finite dimension. We  named this dynamical system the \textit{controlled quadratic gradient flow}.
Finally, we introduce what we call  the \textit{controlled quadratic gradient descent} and the \textit{controlled proximity operator} which are respectively  the Euler explicit  and  implicit discretization  of the  controlled  gradient flow.  
\end{abstract}

\input{parts/sec_introduction}
\input{parts/sec_notations}

\input{parts/sec_opt-ctrl}

\input{parts/sec_ctrl-prox}

\input{parts/sec_applications}

\input{parts/sec_conclusions}

\input{parts/sec_appendix}

\bibliographystyle{plain}
{
\bibliography{ref}
}

\end{document}

%% file: parts/sec_introduction.tex
\section{Introduction}\label{sec:introduction}
\subsection{Problem statement \& Motivations}

In this paper, we consider  the following quadratic optimization problem  
\begin{equation}\label{eq:quadratic-intro}\tag{$\calP$}
    \min_{x\in\bbR^n} f(x)=\frac{1}{2}\pscal{x,Ax}+\pscal{b,x}+c,
\end{equation}
where  \eqref{eq:quadratic-intro} obey to the following assumptions
\begin{assumption}$~$
\begin{itemize}
\item $\bbR$ is the set of real vectors, $n, m$ are positive integers, 
\item $A\in\bbR^{n\times n}$ is a symmetric, positive semidefinite, $b\in\bbR^n$ a vector.
\end{itemize}
\end{assumption}
 \eqref{eq:quadratic-intro} is  the one of the most famous approach to solve problems arising in  many fields  like  operational research, inverse problems, machine learning and automatic differentiation and has applications in many real life problem such as  X-ray crystallography and light scattering.
 For instance, the least-square formulation of the compressed sensing can be written as: 
\begin{equation}\label{eq:least-square}
    \min_{x\in\bbR^n}\frac{1}{2}\normm{Bx-y}^2=\frac{1}{2}\pscal{x,\transp{B}Bx}+\pscal{-\transp{B}y,x}+\normm{y}^2.
\end{equation}
where we want to recover the signal $\avx\in\bbR^n$ from  the measurements vector  $y=B\avx$. In such case,  one  can identify that $A=\transp{B}B, b=-\transp{B}y$ and $c=\normm{y}^2.$ 
The set of critical points of \eqref{eq:quadratic-intro} is given by  
\[
\crit(f)=\Ba{x\in\bbR^n \qsubjq Ax+b=0},
\]
it is a set of linear equations.

\subsection{Prior work and Contributions}
\paragraph{Optmization}$~$

Optimization is an important discipline of mathematics with many applications in real-world problems, see for instance the encyclopedia of optimization \cite{panos2009}.
To optimize $f$ and then solve \eqref{eq:quadratic-intro}, the basic and intuitive idea is to  make a search for the optimal value using the value function $f$ . One can see this has a ``Rollercoaster approach'', it  consists in finding  the minimum  by using only $f$. It means checking all the value of $f$ to find the minimum. This method is  easy to implement on simple or very basic function but for more sophisticated function, it is easy to see that  it is an NP-hard problem. In view to improve the seek, one can use more information on $f$ such as first-order derivatives. The pioneered work of Cauchy \cite{cauchy1847}, based on following the negative of the  gradient and is called the  \textit{gradient-flow}. Before Cauchy, we have the work in Latin by Newton \cite{newton_60} in 1760, for finding the zeros of a function $f$ and can be applied to the gradient of $f$, yielding to Newton methods.The convergence of the Newton methods has been studied in works such as   \cite{kantorovich_49,rall1966}. Concerning the gradient flow, it convergence has been intensively studied by \cite{rockafellar_variational_1998,rockafellar_convex_1970, mordukhovich_variational_2018,brezis78,attouchfadili2023}. Let us now turn our attention to the gradient flow of this problem, we have 
\begin{equation}\label{eq:quad_gradflow}\tag{grad-flow}
    \dot{x}(t)=-Ax(t)-b.
\end{equation}
This equation ensure to find an element of the critical points of $f$. It has been shown for convex function, for instance by Br\'ezis, Bruck and Haddad, that the  trajectories weakly converges to the set of $\argmin(f)$.
Other works try to solve the optimization problem using several approach such as inertial scheme: Polyak methods in \cite{polyak_63,polyak_64}, Nesterov methods \cite{nesterov_18} , the heavy ball \cite{attouch_00, attouch_18} with appropriate convergences.
\paragraph{Control theory}$~$

 At this step, our main inspiration comes from control theory \cite{control2007,Raissi2019,trelat2005,Iri77}. Roughly speaking, control theory answers the question of whether it is possible or not to control the trajectories of a dynamical system between two known states. Our goal is to control the trajectories of the quadratic gradient flow by acting on  this system with a  another ``function'' $u$  commonly  called \textit{control}. Why  control theory in optimization? because the gradient flow only allow to move in the direction of the gradient, our main purpose here is to act independently on the system to ensure that the system converges by allowing different kind of direction on the gradient flow system.
We call our new approach \textit{controlled optimization}. One can also observe that every other methods such as the set of inertial methods are actually some kind of control over the gradient flow with an appropriate choice of direction (control) over the system.
To the best of our knowledge the closest approach to  ours is this paper \cite{recht_tour_2018} where the author proposed a model of reinforcement learning where we have an objective function subject to a constraint  which is a controlled dynamical system.  The two approaches are completely different because we impose the control directly on the gradient-flow dynamics.

\subsection{Outline of the paper}
In Section~\ref{sec:ctrl-opt}, we develop our main theory of  controlled-optimization , where we study the controllability of the gradient-flow  for minimizing a quadratic function then we discretize the corresponding system using an Euler explicit scheme and we find the quadratic gradient descent for solving the minimization of quadratic function. In Section~\ref{sec:ctrl-prox}, we focus on properly define the controlled Moreau's proximity operator. In Section~\ref{num_expe}, we  apply our controlled gradient descent to solve an example compressed sensing problem for Gaussian measurements.

%% file: parts/sec_notations.tex
\section{Notations and preliminaries}\label{sec:preliminaries}

\paragraph{Vectors and matrices} 
We denote $\pscal{\cdot,\cdot}$ a scalar product on $\bbR^n$ and $\normm{\cdot}$ the corresponding norm. $B(x,r)$ is the corresponding ball of \tcb{radius $r$} centered at $x$ and $\mathbb{S}^{n-1}$ is the corresponding unit sphere. Moreover, $\norm{\cdot}{p}, p\in[1,\infty]$ stands for the $\ell_p$ norm. \tcb{ We denote by $\Gamma_0(\bbR^n)$ the class of proper lsc convex function on $\bbR^n$. }    Given a matrix $M \in \bbR^{m \times n}$, $\transp{M}$ is its transpose.

\paragraph{Functions}
A function  $f: \bbR^n \to \barR = \bbR \cup \{\pinf\}$ is closed (or lower semicontinuous, lsc) if its epigraph is closed. The effective domain of $f$ is $\dom(f) = \enscond{x\in\bbR^n}{f(x) < +\infty}$ and $f$ is proper if $\dom(f) \neq \emptyset$ as is the case when it is finite-valued. 

 The Legendre-Fenchel conjugate of $f$ is $f^*(z)=\sup\limits_{x \in \bbR^n} \pscal{z,x} - f(x)$.

We will need to properly define the exponential of a matrix or a linear operator over $\bbR^n$ 
\begin{equation}
e^A\eqdef\sum_{k=0}^{\infty}\frac{A^k}{k !}, \qwhereq A^k=\underbrace{A\times A\times\ldots\times A}_{\text{$k$ times}},
\end{equation}
Let us define for any $\alpha\in\bbR^{n}$ and $x\in\bbR$ we define the following polynoms
\[
P_{\alpha}(x)={\alpha}_0+\alpha_1x+\ldots+\alpha_n x^n.
\]
\begin{theorem}[Hamilton-Cayley]\label{thm:hamilton-cayley} For any $A\in\bbR^{n\times n}$ be a linear operator then their exists $a\in\bbR^{n+1}$ such that
\[
P_a(A)=0.
\]
\end{theorem}

%% file: parts/sec_opt-ctrl.tex
\section{Controlled Optimization}\label{sec:ctrl-opt}
We  consider in this work \eqref{eq:quad_gradflow} on which we act with an additive control. We have
\begin{equation}\label{eq:control-quad}
    \dot{x}(t)=-Ax(t)-b+Bu(t),
\end{equation}
where $B\in\bbR^{n\times m}$ with $m\leq n$ and $u$ is the control taking values in $\bbR^m$.  Let us simplify the notations, we get 
\begin{equation}\label{eq:control-quad-0}
    \dot{x}=-Ax+Bu-b.
\end{equation}
\eqref{eq:control-quad-0} is also known in the control theory literature as the \textit{time invariant linear control system}.
Let $T_0$ and $T_1$ be two real numbers such that $T_0<T_1$. We want to study the following controlled system 
\begin{equation}\label{eq:control-quad-init}\tag{cq-Gf}
\begin{cases}
   \dot{x}(t)=-Ax+Bu-b\\
   x(T_0)=\xo\in\bbR^n, u\in L^{\infty}((T_0,T_1);\bbR^m).  
\end{cases}
\end{equation}
\subsection{Existence of solution}\label{eq:existence-solution}
Let us consider the system \eqref{eq:control-quad-init}, we have the following existence result. 
\begin{proposition}[Existence result]
For any $\xo\in\bbR^n$, one  has that the solution to \eqref{eq:control-quad-init} is unique.
    \end{proposition}
\begin{proof}
    Let $x(t)$ and $y(t)$ two differents curves generated by the controlled gradient flow starting at the same point $x_0\in\bbR^n$, we consider the following distance  
    \[
    g(t)=\frac{1}{2}\normm{x(t)-y(t)}^2, 
    \]
  therefore 
  \begin{align*}
      g'(t)&=\pscal{\dot{x}(t)-\dot{y}(t),x(t)-y(t)}\\
      &=\pscal{-Ax(t)-b+Bu(t)+Ay(t)+b-Bu(t),x(t)-y(t)}\\
      &=-\pscal{Ax(t)+b-Ay(t)-b,x(t)-y(t)}\geq0.
  \end{align*}
  This implis that $g'(t)\leq0$ thus $g$ is a decreasing function of $t$ thus 
  $$
  \forall t>0,\quad g(t)\leq g(0)=0.
  $$ 
\end{proof}

\subsection{Controllability of the gradient flow for quadratic functions}\label{sec:ctrl-gf}
Let us properly define when a system is called ``controllable'' in finite dimension. 

\begin{definition}$~$
\begin{itemize}
\item \textbf{(Controllability)}  Let us consider the system \eqref{eq:control-quad-init}, we say that  \eqref{eq:control-quad-init} is  controllable from $\xo\in\bbR^n$ in time $T>0$ if their exist $\tilde{T}\in[T_0,T_1]$ such that $T= \tilde{T}-T_0$ and  one can reach from the initial  point $x(T_0)=\xo$  any point of the space by acting on the system with a control $u\in L^{\infty}((T_0,T_1);\bbR^m)$ .

\item \textbf{(Attainable  from a set)} Let $\calC$ be a subset of $\bbR^n$. 
The set that one can attain from $\calC$  denoted $\calA[\calC]$ is the set of all points $x(T_1)$ where $x$ is an admissible trajectory to some control $u$ and where $x(T_0)\in\calC$.
\end{itemize} 
\end{definition}

\begin{remark}$~$
\begin{itemize}
	\item  From a control theory perspective, a dynamical system is said to be \textit{controllable} if $\calA[\bbR^n]=\bbR^n.$ Moreover, we can also say that a dynamical system is said to be \textit{controllable} from $\calC$ if $\calA[\calC]=\bbR^n.$ 
    \item From an optimization perspective, given any  $\calC\subset\bbR^n$ we are only  interested to know whether  $\calA[\calC]\cap\Argmin(f)\neq\emptyset$.  In this case, with a good choice of control $u(\cdot)$, we can solve \eqref{eq:control-quad-init} and we say that the system is \textit{opt-controllable} (optimization-controllable) from the set $\calC$. 
    
  \item  If we only have  that $\calA[\calC]\cap\crit(f)\neq\emptyset$ we will simply say that the system is \textit{weakly opt-controllable} from the set $\calC.$
    
    \item One can see that \textit{controllability} from a point  implies \textit{opt-controllability} and the latter implies \textit{weak opt-controllability}.
 
    \item In this work, we will be focusing only on showing the controllability of the gradient flow for solving the quadratic optimization problem. 
\end{itemize}
\end{remark}

The following result  is  fundamental in linear control theory. It gives a necessary and sufficient condition for controllability. It is also known as the rank condition.

\begin{theorem}[Controllability of Gradient flow]\label{theo-cond} The \textit{control quadratic Gradient flow} \eqref{eq:control-quad-init} system is controllable on $[T_0,T_1]$ if and only if we have 
\begin{equation}\label{eq:Kalman-cond}\tag{\small{Pontryagin-Kalman} cond}
    \Span\Bba{(-A)^iBy; y\in \bbR^m, i=0,\cdots,n-1}=\bbR^n.
\end{equation}
\end{theorem}
\begin{proof}
We report the proof to the Appendix Section~\ref{pf:theo-cond}.
\end{proof}
\begin{remark}$~$
    \begin{itemize}
        \item Theorem~\ref{theo-cond} states that under the \eqref{eq:Kalman-cond}, the system \eqref{eq:control-quad-init} is controllable. Roughly speaking, the gradient flow for minimizing the function $f(x)=\frac{1}{2}\pscal{x,Ax}+\pscal{b,x}+c$ is controllable. Therefore,  given an initial point $x(T_0)=\xo$  and a desired target   $x(T_1)=\xsol_d\in\crit(f)$ one can control gradient flow from $\xo$ to $\xsol_d$. This result is quite important since it means that instead of running the gradient flow to find any element in  $\crit(f)$, one can control the gradient flow in such a way that we can reach our desired state $\xsol_d$.

        \item A similar condition \eqref{eq:Kalman-cond}  called the Kalman condition appears first in a paper by Lev. Pontryagin in 1959 and later in a joint paper by Rudolph Kalman \etal in 1963, it also known as the Kalman condition in control theory but we think the name \textit{Pontryagin-Kalman} is more suitable for optimization.
        \item The Pontryagin-Kalman condition or controllability condition can be equivalently stated as a rank condition \ie
        \begin{equation}\tag{Pontryagin-Kalman cond}
            \rank(\scrC)=n\qwhereq \scrC\eqdef\left[B|(-A)B|A^2B\cdots|(-A)^{n-1}B\right].
        \end{equation}
       \item Although controllability of \eqref{eq:control-quad-init} is really interesting, it is not realistic since we can not control a real-world system with any desired control typically with a control which has <<high energy>>.
    \end{itemize}            
\end{remark} 
\paragraph{\textbf{Application to Newton's methods for solving $\calP$}}
 Let us consider that we want to solve the same problem using the Newton's Methods and we hope to get the same  kind of condition. We then consider the following equation 
\begin{equation}\label{eq:control-quad-NM}\tag{quadNM}
\begin{cases}
   A\dot{x}(t)=-Ax(t)-b\\
   x(T_0)=\xo\in\bbR^n,
\end{cases}
\end{equation}
Let us introduce an additive control, similarly to the gradient flow system, we have
\begin{equation}\label{eq:control-quad-NM}\tag{cq-quadNM}
\begin{cases}
   A\dot{x}(t)=-Ax(t)-b+Bu(t)\\
     x(T_0)=x_0\in\bbR^n, u\in L^{\infty}((T_0,T_1);\bbR^m).  
\end{cases}
\end{equation}
\begin{corollary}[Controllability of Newton method]\label{cor:ctrl-nm} Let us consider the problem \eqref{eq:quadratic-intro}, the dynamic  \eqref{eq:control-quad-NM} is controllable if and only if
\begin{equation}\label{eq:calCnm}
\rank(B)=n.   
\end{equation}
\end{corollary}
\begin{proof}
Let us  define $z(t) \eqdef Ax(t)$, we get that 
\begin{equation}\label{eq:control-quad-NM-dom}\tag{cq-NMd}
\begin{cases}
   \dot{z}(t)=-z(t)-b+Bu(t)\\
     z(T_0)=z_0\in\bbR^n, u\in L^{\infty}((T_0,T_1);\bbR^m)
\end{cases}
\end{equation}
Then, we apply  Theorem~\ref{theo-cond} to have the following condition
\[
\rank(\calC_{NM})=n\qwhereq \calC_{NM}\eqdef\left[B|(-1)B|B|\cdots|(-1)^{n-1}B\right].
\]
Therefore, $\rank(\calC_{NM})=n$ if and only if $\rank(B)=n$.  
\end{proof}

\begin{remark} $~$
\begin{itemize}
\item One can see \eqref{eq:control-quad-NM-dom} as a gradient flow in the domain of the linear operator. For instance,  the Fourier Transform we have a gradient flow in the frequency  domain.
\item We observe that Pontryagin-Kalman condition for Newton method \eqref{eq:calCnm} only depends on the linear operator $B$.
\item If $A$ is invertible then the Pontryagin-Kalman condition for Newton method  can be written as 
\[
\rank(\calC_{NM})=n\qwhereq \calC_{NM}\eqdef\left[A^{-1}B|(-1)A^{-1}B|A^{-1}B|\cdots|(-1)^{n-1}A^{-1}B\right],
\]
indeed the Newton's methods \eqref{eq:control-quad-NM} becomes
\begin{equation}
    \dot{x}(t)=-x(t)-A^{-1}b+A^{-1}Bu.
\end{equation}
We can observe that the rank condition only depends on $B$, since $A$ is invertible.
\end{itemize}
\end{remark}
\paragraph{\textbf{Observations on Stochastic methods}}
From a continuous point of view, we also observe that the Langevin  Monte-Carlo equation applied to solve \eqref{eq:quadratic-intro} and known has the continuous version of the Stochastic gradient descent (SGD) is given by 
\[
dX_t=-(AX_t+b)dt+\sqrt{2}dB_t.
\]
Let us remark that the second term can be seen has a stochastic control over the dynamical systems. We leave the study of the controllability of this equation as a perspective.

\begin{definition}[Feedback control]\label{lem:expres-cont}
Let us consider a controlled sytem \eqref{eq:control-quad-init}, we define a \textit{feedback control} has a control $u$ such that there exists a linear operator $K:\bbR^n\to\bbR^m$, 
\begin{equation}\label{eq:choice-u-2}
u(t)\eqdef Kx(t)=Kx_t,
\end{equation}
where $K\in\bbR^{m\times n}$ is a constant matrice. We substitute the control expression and we get the system 
\begin{equation}\label{eq:cqgf}
\dot{x}(t)=-Ax(t)+Bu(t)-b=\Ppa{-A+BK}x(t)-b.
\end{equation}
\end{definition}

\subsection{Convergence of the objective function}\label{sec:label-conv-control}
\begin{proposition}[Behavior of the value function]\label{prop:f-behavior}
Let us consider \eqref{eq:control-quad-init}, we set  $x(t)=x_t$ and the control $u(t)=u_t$,  the following  holds
\begin{equation}\label{eq:ctrl-vf}
\dfrac{d}{dt}f(x_t)=-\normm{Ax_t+b}^2 +\pscal{Ax_t+b; Bu_t}.
\end{equation}
\end{proposition}
\begin{proof}
The proof can be found in Appendix~\ref{proof:prop:f-behavior}.
\end{proof}
\begin{remark}$~$
\begin{itemize}
\item Under the feedback control \ie replacing \eqref{eq:choice-u-2} in \eqref{eq:ctrl-vf},  we get 
\begin{equation}
\dfrac{d}{dt}f(x_t)=-\normm{Ax_t+b}^2-\norm{x_t}{\transp{A}BK}^2 +\pscal{b;BKx_t}
\end{equation}
\item This expressions explicitly show how the choice of control affects the descent behavior of $f$. Indeed,  we can observe that $f$ is not only decreasing along the trajectory. The derivative of $f$ with respect to $t$ along the trajectory can be  split in two parts. The first  is the usual negative of the norm of the gradient which ensure that $f$ is decreasing along the trajectory and the second is a term which involve the control $u_t$. Thanks to the control term, $f$ is not only decreasing along the trajectory but also can have different behaviors (acceleration, decceleration,\etc)this allow to easily attain the pre-image of any value of the function in the space up to a good choice of control.
\end{itemize}
\end{remark}

\subsection{Explicit Discretization and convergence rates}\label{sec:explicit-discretization}
The forward Euler method or the explicit discretization of \eqref{eq:control-quad-init} yield to 
\begin{equation}
   \forall k\in\bbN,\quad  \frac{\xkp-\xk}{\gak}=-A\xk-b +Bu_k \Leftrightarrow \xkp=\xk-\gak\Ppa{A\xk+b}+Bu_k,
\end{equation}
where $\xk=x(t_k), u_k=u(t_k)$ with $t_k=k\gak$ and $\gak$ is the stepsize of the discretization of the time interval $[T_0,T_1].$
\begin{definition}[Controlled gradient descent]
We call  the system \eqref{eq:control-q} the controlled gradient descent applied on the quadratic  program \eqref{eq:quadratic-intro}. 
\begin{equation}\label{eq:control-q}\tag{cq-Gd}
\forall k\in\bbN,\quad
\begin{cases}
   \xkp=\xk-\gamma\Ppa{A\xk+b}+Bu_k\\
   \xo\in\bbR^n, u_k\in U\subseteq \bbR^m, B\in\bbR^{n\times m},  
\end{cases}
\end{equation}    
\end{definition}

\begin{proposition}[Convergence rates]\label{prop:conv-rate}Let us define $\tau\eqdef\normm{A-BK}<1$  and consider the \textit{controlled gradient descent} \eqref{eq:control-q} applied to solve the quadratic minimization \eqref{eq:quadratic-intro} and $\xsol\in\argmin(f)$, we have the following convergence rates
\begin{equation}
\normm{\xkp-\xsol}\leq\Ppa{1-\gak\normm{A}}^k\normm{\xo-\xsol}+\normm{B}\sum_{i=0}^{k}(1-\gak\normm{A})^{k-i}\normm{u_i-u_{\star}}
\end{equation}

Under the Feedback control,  one simply has that 
\begin{equation}
\normm{\xkp-\xsol}\leq\Ppa{1-\gak\tau}^k\normm{\xo-\xsol}, \qwhereq \normm{A-BK}=\tau.
\end{equation}
\end{proposition}
\begin{proof}
We leave the proof  in the Appendix~\ref{proof:prop:conv-rate}.
\end{proof}

%% file: parts/sec_ctrl-prox.tex
\subsection{Controlled Proximity operator}\label{sec:ctrl-prox}
\subsubsection*{Implicit discretization on $x$ and Explicit on $u$}\label{sec:impl-explicit}
The backward Euler method or the implicit discretization of \eqref{eq:control-quad-init} yield to
\begin{equation}
   \forall k\in\bbN,\quad \frac{\xkp-\xk}{\gak}=-A\xkp-b +Bu_{k} \Leftrightarrow \xkp+\gak A\xkp+b=\xk + \gak Bu_{k},
\end{equation}
where $\xk=x(t_k), u_k=u(t_k)$ with $t_k=k\gak$ and $\gak$ is the stepsize of the discretization of the time interval $[T_0,T_1].$  We recall that $\nabla f (x)=Ax+b$. We get the following line
$$
\xkp+\gak\nabla f(\xkp)=\xk+\gak Bu_{k}
$$
\begin{equation}\label{eq:control-q-res}\tag{c-resol}
\forall k\in\bbN,\quad
\begin{cases}
   \xkp=\Ppa{\Id+\gak\Ba{\nabla f}}^{-1}\circ \psi_{u_{k}}\pa{\xk}\\
   \xo\in\bbR^n, u_k\in\bbR^m, B\in\bbR^{n\times m},  
\end{cases}
\qwhereq  \psi_{u_{k}}\pa{\xk}=\xk+\gak Bu_{k}
\end{equation}

\begin{remark}$~$

\begin{itemize}

\item This expression \eqref{eq:control-q-res} is closely related to a controlled version of the proximity operator.
\item This \textit{controlled  resolvent}  has probably a link with \textit{the  Bregman D-prox}.
\end{itemize}
\end{remark}

\subsubsection*{Controlled Proximity operator}\label{sec:ctrl-prox-op}
The Proximity operator  is defined in the seminal works by Moreau in \cite{moreau_62,moreau_63,moreau_65}  which appear to be very useful when working with splitting methods. In this subsection, we would like to simply define a controlled version of the proximity operator.
\begin{definition}\label{df:df-cprox}Let us define the following mapping that we call the controlled proximity operator for any control $u\in\bbR^{m},$
\begin{equation}\label{eq:df-cprox}
\mathrm{c-}\prox_{\gamma f}(z)=\argmin_{x\in\bbR^n}\Big\{f(x)+\frac{1}{2\gamma}\normm{x-z}^2-\pscal{Bu;x}\Big\}, 
\end{equation}
\end{definition}

We have the following Proposition which links the resolvant and the the proximity operator. 
\begin{proposition}[Controlled resolvant and proximity operators]\label{prop:ctrl-res} 
Let us consider the controlled proximity operator defined by \eqref{eq:df-cprox},  we have that solving \eqref{eq:df-cprox}  is up  to  applying the  controlled resolvent \eqref{eq:control-q-res}.
\end{proposition}
\begin{proof}
We put the proof in the Appendix~\ref{proof:prop:ctrl-res}.
\end{proof}

%% file: parts/sec_applications.tex
\section{Numerical experiments}\label{num_expe}

We consider a classical problem  which can be put in quadratic form. We  solve it  use the  controlled gradient descent and the classical gradient descent and we compare the numerical result. 

\paragraph{\textbf{Application to Compressive Sensing (CS)}}
In Compressed Sensing, we aim to solve the following minimizing
\begin{equation}\tag{$\calP$}\label{eq:solve-P}
   \min_{x\in\bbR^n} f_{cs}(x)=\frac{1}{2m}\normm{Ax-y}^2,
\end{equation}
here the matrice $A$ represents differents type of linear operator:  Gaussian measurements,  Fourier basis, Wavelet basis, or any Parseval tight frame \cite{stark_sparse_1970}. We use the following scheme 
\begin{align*}
   \begin{cases}
    \xkp=\xk-\gak\nabla f_{cs}(\xk)+Bu_k\\
     \xo\in\bbR^n,B\in\bbR^{n\times d}, u_k\in\bbR^d.
    \end{cases}
\end{align*}
Here we have that 
$
\forall x\in\bbR^n, \nabla f_{cs}(x)= \transp{A}(Ax -y),
$
thus we get the following algorithms:
 \begin{algorithm}[htbp]
	\caption{Controlled gradient descent for CS}
	\label{eq:inertial-methods}
	\textbf{Parameters:}  $0<\gamma<\frac{2}{L}$\;
	\textbf{Initialization:} $\xo\in\bbR^n$ $B\in\bbR^{n\times d}, u_0\in\bbR^d$,  \;
	\For{$k=0,1,\ldots$}{
		\vspace{-0.25cm}
		\begin{flalign} \tag{cgd}\label{cgdalgo} 
			&\begin{aligned} 
				&\xkp=\xk-\gamma \transp{A}(A\xk -y)+Bu_k\\
				&u_{k+1}=\mathrm{updatecontrol}(u_k),
			\end{aligned}&
		\end{flalign}
	}
\end{algorithm}

This is the controlled gradient descent applied to the Compressed sensing problem. Let us consider that $\forall r\in\bbN, (a_r)_{r\in[m]}$ is a standard Gaussian measurement.  The next result show the Pontryagin Kalman condition for Gaussian sensing under sufficiently large number of measurements.

\begin{lemma}[Controlability condition for Gaussian sensing] Let suppose that $A\in\bbR^{m\times n}, B\in\bbR^{n\times d}$  Gaussian sensing for $n\geq c d$  and $m\geq C d$ where $c, C$ are arbitrary large constant we have 
\[
d\leq \rank(\calC)\leq nd, \qwhereq \calC\eqdef\left[B|-AB|A^2B|\cdots|(-A)^{n-1}B\right]. 
\]
\end{lemma}
\begin{proof}
Let us consider the Gaussian matrix $B$, the proof is a direct application of the rank Theorem and the following fact. Under the bound on $n\geq\frac{16}{\rho^2}d $, we have
\begin{equation}
(1-\rho)\normm{u}\leq\frac{1}{m}\normm{Bu}\leq(1+\rho)\normm{u}.
\end{equation}
This kind of bound is called the Restricted isometry property \cite{candes2006}. This implies by the rank Theorem that $\rank(B)=d$, by using a similar type of bound for $m\geq\frac{16}{\rho^2}d$ we have that $\rank(AB)=d$. We apply a similar agurment recurcively to finally get that $d\leq \rank(\calC)\leq nd.$
\end{proof}
\paragraph{ Set-up and Observations}$~$ 

In this simulation, we want to reconstruct a signal  of size $n=128$ from $m$ differents measurements. The Lipschitz coefficient $L\approx\sqrt{m}$. We initialized our schemes using random  initialization in $\bbR^n$. We  compare  the gradient descent (blue) with the controlled gradient descent (red) with some kind of feedback control in several case: oversampled ($m>n$), sampled ($m\approx n$) and undersampled ($m<n$) regime. In figure~\ref{fig: cs-determined}, we displayed the $\ell_2-$error of reconstruction of the signal: on the left in the oversampled regime and on the right in the sampled regime. In the oversampled regime both methods behaves almost the same and one can observe a linear convergence rate. When $m\approx n$, both algorithms behave differently, we can observe that the controlled gradient descent behaves better than the gradient descent. The error with the controlled gradient descent is smaller than the error reconstruction with the gradient descent. Almost the same observation can be made in the undersampled regime see Figure~\ref{fig: cs-underdetermined}, in which case the gradient descent behaves very bad  while with the controlled gradient descent one can expect to have a better reconstruction.  It is not an easy task to find the proper control adapted  to the problem that we want to solve. This choice is also linked with the Landscape of the objective function that we want to minimize. In compressed sensing, the landscape obvioulsy depend on the regime.
\begin{figure}[h]
\centering

\includegraphics[trim={0cm 0cm 0cm 0cm},clip,width=0.49\textwidth]{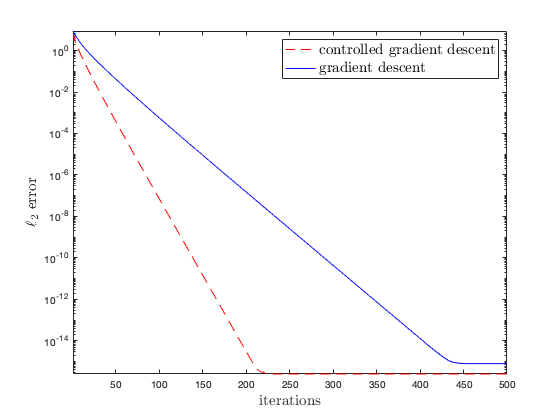}
\includegraphics[trim={0cm 0cm 0cm 0cm},clip,width=0.5\textwidth]{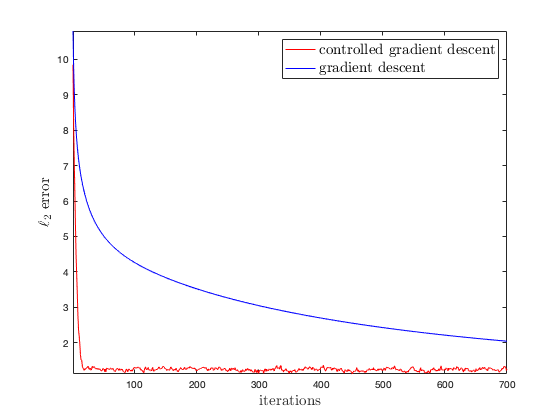}
 \caption{On the left, error of reconstruction of a signal in the oversampled regime, on the right  side error of the reconstruction of a signal  when $m\approx n$.}
\label{fig: cs-determined}
\end{figure} 
\begin{figure}[htbp]
\centering
\includegraphics[trim={0cm 0cm 0cm 0cm},clip,width=0.48\textwidth]{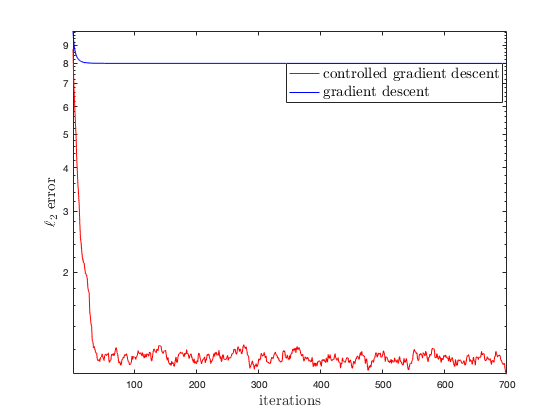}
 \caption{Error recontruction of a signal when $m=0.5\times n$ (undersampled regime).}
\label{fig: cs-underdetermined}
\end{figure}

%% file: parts/sec_conclusions.tex
\section*{Conclusion \& Perspectives}
In this work, we have introduced the \textit{controlled gradient flow} which is a new perspective to optimization and we have applied our result to solve the the Compress sensing problem which is a linear inverse problem. This work allows us to see optimization  and many  algorithm from another point of view. One can see several existing algorithms has control over the gradient flow. The main problem is the choice of the control which of course depend on the objective function and the choice of the direction.

%% file: parts/sec_appendix.tex
\begin{appendix}
\section{Proofs for Section~\ref{sec:ctrl-opt}}
\subsection{Proofs of Theorem~\ref{theo-cond}}\label{pf:theo-cond}

Let us consider the system \eqref{eq:control-quad-init}, we have
$$
\dot{x}(t) = -A x(t) + B u(t) + b, \quad x(T_0) = x_0,
$$
where \( A \in \mathbb{R}^{n \times n} \), \( B \in \mathbb{R}^{n \times m} \), and \( u(t) \in \mathbb{R}^m \). We put
\[
y(t)\eqdef B u(t) + b \in \mathbb{R}^n.
\]

Then the system becomes
\begin{equation}\label{eq:lin-ctrl}
\begin{cases}
   \dot{x}=-Ax(t)+y(t)\\
   x(T_0)=\xo\in\bbR^n, y\in \Im(B(U))+b\subset \bbR^n.
\end{cases}
\end{equation}

We have made a change of  variable over the control $u\in U$ to get a new control $y\in\bbR^n$.  To acheive this goal, we consider the homgeneous system which is given by 
\begin{equation}\label{eq:hgsystem}
\begin{cases}
   \dot{x}=-Ax(t)\\
   x(T_0)=\xo\in\bbR^n,
\end{cases}
\end{equation}
It is simple to verify that a solution to \eqref{eq:hgsystem} is given by the following expression
\[
\forall t \in [T_0,T_1], \quad\lambda(t)=  e^{-A(t-T_0)}\xo. 
\]
where $ e^{-A(t-T_0)}$ is the resolvant from a dynamical perspective and the basic of the semi-group of contraction.

Finally, the solution to the heterogeneous system is given by the following 
\begin{equation}\label{eq:sol-het}
\forall t \in [T_0,T_1], \quad x(t)=  e^{-A(t-T_0)}\xo+\int_{t_0}^te^{-A(s-T_0)}y(s)ds.
\end{equation}
To steer the system to a desired final state \( x_d \in \bbR^n \) at time \( T_1 \), we require that
\[
x(T_1) = x_d \quad \Longleftrightarrow \quad \int_{T_0}^{T_1} e^{-A(s-T_0)} y(s)ds = x_d - e^{-A(T_1 - T_0)} x_0.
\]
For any $y\in \Im(B(U))+b$ implies that $y\in L^{\infty}([T_0,T_1], \bbR^n)$ since $B\in\bbR^{n\times m}$. Now, we define the linear operator
$$
\varphi : L^\infty([T_0, T_1], \mathbb{R}^n) \to \mathbb{R}^n, \quad
\varphi(y)\eqdef \int_{T_0}^{T_1} e^{-A(s-T_0)} y(s) ds.
$$

To ensure that any desired state $ x_d $ is reachable from $ x_0 $, the operator $ \varphi $ must be surjective. We now connect this to the controllability condition. Define the controllability matrix or Pontryagin-Kalman condition
$$
\scrC := [B \mid -AB \mid A^2 B \mid \cdots \mid (-A)^{n-1} B] \in \mathbb{R}^{n \times nm}.
$$
% We have  \textit{$\rank(PK)=n$  if and only if  the previous mapping is surjective}. 
 
 Indeed, on one hand  we have that if $\rank(\scrC)<n$ the previous mapping $\varphi$ is not surjective. The fact that $\rank(\scrC)<n$ means that the the matrice $\calC$ is not surjective. There exists $\alpha\in\bbR^n\backslash\{0\}$  such that $\alpha \scrC=0$ thus \[\alpha B=-\alpha AB=A^2B=\ldots=(-\alpha)^{n-1}A^{n-1}B=0.\]
Now we apply the Hamilton-Cayley Theorem~\eqref{thm:hamilton-cayley} and we have
\[
(-A)^{n}= \sum_{j=0}^{n-1}a_iA^{i}
\]
Using now the expression of the exponential of a matrice and similar on $b$ and we get that
\begin{align*}
\forall k\in \bbN, (-1)^{n-1}\alpha A^{k}B=0&\Longrightarrow  (-1)^{n-1}\alpha A^{k}y=0\\
&\Longrightarrow  \alpha e^{-(t-t_0)}y=0\\
&\Longrightarrow  \alpha \int_{t_0}^te^{-A(t-T_0)}ydt=0.
\end{align*}
Therefore  the  previous corresponding mapping is surjective. On the other hand  if $\varphi$ is not surjective thus similarly it is straightforward that $\rank(\scrC)<n$. Since $\varphi$ is surjective
we get that \eqref{eq:sol-het} is can attain any point in $\bbR^n$ implying that the system is controllable.

\subsection{Proof of Proposition~\ref{prop:f-behavior}}\label{proof:prop:f-behavior}
By the Chain rule, we have that 
\begin{align*}
\frac{d}{dt}f(x_t)&=\pscal{\nabla f(x_t); \dot{x}_t}=\pscal{\nabla f(x_t); -\nabla f(x_t)+Bu_t},\\
&=-\normm{Ax_t+b}^2 +\pscal{Ax_t+b; Bu_t},\\
&=-\normm{Ax_t+b}^2 +\pscal{\Ppa{Ax_t+b}; Bu_t},\\
&=-\normm{Ax_t+b}^2 +\pscal{x_t; ABu_t}+\pscal{b; Bu_t}.
\end{align*}
Taking $u_t=Kx_t$, one get that
\begin{align*}
\frac{d}{dt}f(x_t)&=-\normm{Ax_t+b}^2 +\pscal{x_t; ABKx_t}+\pscal{b; BKx_t},\\
&=-\normm{Ax_t+b}^2 +\norm{x_t}{ABK}^2+\pscal{b; BKx_t}.
\end{align*}

\subsection{Proof of Proposition~\ref{prop:conv-rate}}\label{proof:prop:conv-rate}
Let consider a sequence $\seq{\xk}$ generated by the controlled gradient descent and $\xsol$ a solution. If the sequence $\seq{\xk}$ converges toward a  solution $\xsol$, it means that it exists a control let call it $\uustar$ yielding to $\xsol$. Therefore, we have 
\begin{align*}
\xkp-\xsol&=\xk-\gak A\xk +b+Bu_k-\xsol+\gak A\xsol -b-B\uustar,\\
&=\Ppa{\xk-\xsol}-\gak\Ppa{A\xk-A\xsol}+B\pa{u_k-\uustar},
\end{align*} 
Taking now the norm and thanks to the triangular inequality we get 
\begin{align*}
\normm{\xkp-\xsol}\leq \Ppa{1-\gak \normm{A}}\normm{\xk-\xsol}+ \normm{B}\normm{u_k-\uustar},\\
\end{align*}
Recursively and after summing, we get that 
\begin{align*}
\normm{\xkp-\xsol}\leq \Ppa{1-\gak \normm{A}}^k\normm{\xk-\xsol}+ \normm{B}\sum_{i=0}^{k} \Ppa{1-\gak \normm{A}}^{k-i}\normm{u_i-\uustar},\\
\end{align*}
In the particular case, if we use the control $u=Kx$ one get,
\begin{align*}
\xkp-\xsol&=\xk-\xsol-A\Ppa{\xk-\xsol}+ BK\Ppa{\xk-\xsol}\\
&=\xk-\xsol-\Ppa{A-BK}\Ppa{\xk-\xsol}
\end{align*}
thus we get that
\begin{align*}
\normm{\xkp-\xsol}&\leq{\normm{\Id-\pa{A-BK}}} \normm{\Ppa{\xk-\xsol}}\\
\Longrightarrow \normm{\xkp-\xsol}&\leq\normm{\Id-\pa{A-BK}}^k \normm{\Ppa{\xo-\xsol}}\\
\end{align*}
where $\normm{\Id-\pa{A-BK}}$ is the spectral radius of the corresponding matrice. One can bound the latter  and have that 
\begin{align*}
\normm{\xkp-\xsol}\leq{\Ppa{1-\normm{A-BK}}^k}\normm{\xk-\xsol}
\end{align*}
A good condition which ensure linear convergence of the sequence is  $\normm{A-BK}<1$.
\section{Proof for Section~\ref{sec:ctrl-prox}}
\subsection{Proof of Proposition~\ref{prop:ctrl-res}}\label{proof:prop:ctrl-res}
Let us consider the optimization problem \eqref{eq:df-cprox} ,  $f$ is a quadratic form over $\bbR^n$  and $A$ is symmetric positive definite which implies that $f\in\Gamma_0(\bbR^n)$ moreover $\forall u\in U, \quad x\mapsto \pscal{Bu;x}$ is a linear form  thus $f(x)-\pscal{Bu;x}$ is continuous proper and convex over $\bbR^n$. Therefore our function is strongly convex and have a unique minimizer. For any fixed  $z\in\bbR^n$, let $\Ba{\xsol}=\mathrm{c-}\prox_{ f}(z)$
The first-order optimality condition yields to:
\begin{align*}
\nabla f(\xsol)+(\xsol-z)-Bu=0 &\Longleftrightarrow \nabla f(\xsol)+\xsol=z+Bu,\\
&\Longleftrightarrow (\{ \nabla f\}+\Id)\xsol=z+ Bu,\\
&\Longleftrightarrow \xsol=(\{ \nabla f\}+\Id)^{-1}\Ppa{z+ Bu}.
\end{align*}
\end{appendix}

%% file: lncs_opt-ctrl.bbl
\begin{thebibliography}{10}

\bibitem{cauchy1847}
Cauchy A.
\newblock Méthode générale pour la résolution des systèmes d’équations
  simultanées.
\newblock {\em Comptes Rendus de l’Académie des Sciences de Paris},
  25:536–538, January 1847.

\bibitem{attouchfadili2023}
H.~Attouch and J.~Fadili.
\newblock From the ravine method to the nesterov method and vice versa: A
  dynamical system perspective.
\newblock {\em SIAM Journal on Optimization}, 32(3):2074--2101, 2022.

\bibitem{attouch_00}
H.~Attouch, X.~Goudou, and P.~Redont.
\newblock The heavy-ball with friction method, i. the continuous dynamical
  system: global exploration of the local minima of a real-valued function by
  asymptotic analysis of a dissipative dynamical system.
\newblock {\em Communications in Contemporary Mathematics}, 2(01):1--34, 2000.

\bibitem{attouch_18}
H.~Attouch and J.~Peypouquet.
\newblock The rate of convergence of nesterov’s accelerated forward-backward
  method is actually faster than $o(k^{-2})$.
\newblock {\em SIAM Journal on Optimization}, 26(3):1824--1834, 2016.

\bibitem{brezis78}
H.~Brezis.
\newblock Asymptotic behaviour of some evolution systems.
\newblock In Michael~G. Crandall, editor, {\em Nonlinear Evolution Equations},
  pages 141--154. Academic Press, 1978.

\bibitem{candes2006}
E.J. Candes, J.~Romberg, and T.~Tao.
\newblock Robust uncertainty principles: exact signal reconstruction from
  highly incomplete frequency information.
\newblock {\em IEEE Transactions on Information Theory}, 52(2):489--509, 2006.

\bibitem{control2007}
J.-M. Coron.
\newblock {\em Control and Nonlinearity}.
\newblock American Mathematical Society, USA, 2007.

\bibitem{panos2009}
C.~A. Floudas and P.~M. Pardalos.
\newblock Encyclopedia of optimization.
\newblock (2):489--509, 2009.

\bibitem{Iri77}
M.~Iri, N.~Tomizawa, and S.~Fujishige.
\newblock On the controllability and observability of a linear system with
  combinatorial constraints.
\newblock {\em Journal of the Society of Instrument and Control Engineers},
  13:235--242, 1977.

\bibitem{kantorovich_49}
L.~Kantorovich.
\newblock On newton’s method.
\newblock {\em Trudy Mat. Inst. Steklov}, 28:104–144, 1949.

\bibitem{mordukhovich_variational_2018}
S.~M. Mordukhovich.
\newblock {\em Variational {Analysis} and Applications}.
\newblock Springer Monographs in Mathematics. Springer Cham, Berlin,
  Heidelberg, 2018.

\bibitem{moreau_62}
J.-J. Moreau.
\newblock Fonctions convexes duales et points proximaux dans un espace
  hilbertiens.
\newblock {\em Comptes rendus hebdomadaires des S\'eances de l'Acad\'emies des
  Sciences}, 255:2897--2899, 1962.

\bibitem{moreau_63}
J.-J. Moreau.
\newblock Propri\'et\'es des applications ``prox''.
\newblock {\em Comptes rendus hebdomadaires des S\'eances de l'Acad\'emies des
  Sciences}, 256:1069--1071, 1962.

\bibitem{moreau_65}
J.-J. Moreau.
\newblock Proximit\'e et dualit\'e dans un espace hilbertiens.
\newblock {\em Bulletin de la Soci\'et\'e Math\'ematiques de France},
  93:273--299, 1965.

\bibitem{nesterov_18}
Y.~Nesterov.
\newblock Lectures on convex optimization.
\newblock {\em Springer}, 137:1–17, 2018.

\bibitem{newton_60}
I.~Newton.
\newblock Philosophiae naturalis principia mathematica.
\newblock {\em Sumptibus et Ant, Philibert.}, 1760.

\bibitem{polyak_63}
B.~T. Polyak.
\newblock Gradient methods for the minimisation of functionals.
\newblock {\em U.S.S.R. Comput. Math. and Math. Phys.}, 3(4):864–878, 1963.

\bibitem{polyak_64}
B.~T. Polyak.
\newblock Some methods of speeding up the convergence of iteration methods.
\newblock {\em U.S.S.R. Comput. Math. and Math. Phys.}, 4(5):1–17, 1964.

\bibitem{Raissi2019}
N.~Raissi, C.~Sanogo, and M.~Serhani.
\newblock Differential game model for sustainability multi-fishery.
\newblock {\em Trends in Biomathematics: Mathematical Modeling for Health,
  Harvesting, and Population Dynamics: Selected works presented at the BIOMAT
  Consortium Lectures, Morocco 2018}, pages 119--137, 2019.

\bibitem{rall1966}
L.~B. Rall.
\newblock Convergence of the newton process to multiple solutions.
\newblock {\em Numer. Math.}, 9:23–37, 1966.

\bibitem{recht_tour_2018}
B.~Recht.
\newblock {\em A tour of reinforcement learning : the view from continuous}.
\newblock arxiv:1806.09460v2, 2018.

\bibitem{rockafellar_convex_1970}
R.~T. Rockafellar.
\newblock {\em Convex Analysis}.
\newblock Princeton University Press, 1970.

\bibitem{rockafellar_variational_1998}
R.~T. Rockafellar and R.~J.~B. Wets.
\newblock {\em Variational {Analysis}}, volume 317 of {\em Grundlehren der
  mathematischen {Wissenschaften}}.
\newblock Springer Berlin Heidelberg, Berlin, Heidelberg, 1998.

\bibitem{stark_sparse_1970}
J.-L. Starck, F.~Murtagh, and J.~M. Fadili.
\newblock {\em Sparse Image and Signal Processing}.
\newblock Cambridge Press., 2015.

\bibitem{trelat2005}
E.~Tr{\'e}lat.
\newblock {\em Contr{\^o}le optimal: th{\'e}orie \& applications}.
\newblock Math{\'e}matiques concr{\`e}tes. Vuibert, 2005.

\end{thebibliography}
